\documentclass[letterpaper,10pt]{article}
\usepackage{amsmath,amsfonts,amsthm,amssymb}
\usepackage{mathrsfs} 
\usepackage{fancyhdr}
\usepackage{pgf}
\usepackage{tikz}
\usetikzlibrary{arrows,automata}
\usetikzlibrary{matrix}
\usetikzlibrary{shapes,backgrounds,calc}
\usepackage{lineno,color}
\usepackage{subfigure}

\newtheorem{theorem}{Theorem}

\newtheorem{proposition}[theorem]{Proposition}

\theoremstyle{remark}

\theoremstyle{definition}

\def \deg {\text{deg}}

\def \cF {\mathcal{F}}

\def \H {\mathcal{H}}
\def \cL {\mathcal{L}}

\def \R {\mathbb{R}}

\def \cT {\mathcal{T}}

\newcommand{\de}{\textup{deg}}
\def \lc {\left\{}
\def \rc {\right\}}
\def \ba {\lc\begin{array}{c}}
\def \ea {\end{array}\rc}

\def \Gsc {1.5} 

\begin{document}
\title{Counting Spanning Trees in Threshold Graphs}
\author{Stephen R. Chestnut \and Donniell E. Fishkind}

\maketitle
\begin{abstract}
Cayley's formula states that there are $n^{n-2}$~spanning~trees in the complete~graph on $n$~vertices; it has been proved in more than a dozen different ways over its $150$~year history. The complete~graphs are a special case of threshold~graphs, and using Merris'~Theorem and the Matrix~Tree~Theorem, there is a strikingly simple formula for counting the number of spanning trees in a threshold graph on $n$~vertices; it is simply the product, over ${i=2,3, \ldots,n-1}$, of the number of vertices of degree~at~least~$i$. In this manuscript, we provide a direct combinatorial proof for this formula which does not use the Matrix~Tree~Theorem; the proof is an extension of Joyal's~proof for Cayley's~formula.  Then we apply this methodology to give a formula for the number of spanning trees in any difference~graph.
\end{abstract}

\section{Overview}

Unless otherwise specified, our graphs are undirected and simple. For any graph~$G=(V,E)$, let~$\tau(G)$ denote the number of spanning trees~in~$G$, and for any vertex~$v \in V$, $\de (v):= | \{u \in V:u \sim v \} |$ is the degree~of~$v$, where~$u \sim v$ indicates that $u$~and~$v$ are adjacent~in~$G$. Let~$K_n$ denote the complete graph on~$n$ vertices. Cayley's formula~\cite{cayley1889theorem} states that, for any positive integer~$n$, $\tau(K_n)=n^{n-2}$; this classical result has been proved in over a dozen different ways over the last $150$~years. We review and discuss this rich history in Section~\ref{sec:disc}.

A graph~$G=(V,E)$ is a \emph{threshold~graph} if there exists a weighting function~${\phi:V \rightarrow \R}$ and a threshold~$\alpha \in \R$ such that each pair of distinct vertices~$u,v$ is adjacent if and only~if ${\phi (u) + \phi (v) \geq \alpha}$. For example, the graph in Figure~\ref{fig:thresholdgraph} is a threshold~graph.  Also, $K_n$ is a threshold~graph since we may choose ${\alpha=0}$ and~${\phi(v)=1}$, for~all~${v\in V}$. The following result is a direct consequence of Merris' Theorem~\cite{merris1994degree} and the Matrix Tree Theorem~\cite{kirchhoff1847ueber}. The result was discovered independently by Bogdanowicz \cite{bogdanowicz1985thesis} and Hammer and Kelmans \cite{hammer1996laplacian}; see also~\cite{bleiler2007correction}.  It has been extended to incorporate the degree sequences of the trees by Martin and Reiner \cite{martin2003factorizations} and extended to more general classes of graphs by Duval, Klivans, and Martin \cite{duval2009simplicial} and Remmel and Williamson \cite{remmel2002spanning}.

\begin{theorem} For any threshold graph~$G=(V,E)$ with $n$ vertices it holds that
$$\tau (G)= \frac{1}{n}\prod_{i=1}^{n-1} | \{ v\in V: \de (v) \geq i  \} |.$$
\label{thm:main}
\end{theorem}
Notice that if $G$ has no isolated vertex then ${|\{v\in V:\de(v)\geq 1\}|=n}$; whereas, if~$G$ has an isolated vertex then ${|\{v\in V:\de(v)\geq n-1\}=0}$, for~${n\geq2}$.  Thus, the formula is equivalent to ${\tau(G)=\prod_{i=2}^{n-1} | \{ v\in V: \de (v) \geq i  \} |}$, for~${n\geq3}$.  

Each of the $n$~vertices in~$K_n$ has degree~$n-1$, so the formula in Theorem~\ref{thm:main} becomes~${\tau(K_n)=n^{n-2}}$, Cayley's~formula. Thus, Theorem~\ref{thm:main} is a generalization of Cayley's formula from the complete graph to all threshold graphs.

In Section~\ref{sec:cons} we review the Matrix Tree~Theorem, Merris'~Theorem, and we review how Theorem~\ref{thm:main} immediately follows from these results. This approach to proving Theorem~\ref{thm:main} is standard~and~elegant, however the very combinatorial simplicity of Theorem~\ref{thm:main} demands a direct combinatorial proof.  

Section~\ref{sec:proper} reviews properties of threshold graphs, and then Section~\ref{sec:comb} presents a direct combinatorial proof of Theorem~\ref{thm:main}.  The proof is an extension of Joyal's proof of Cayley's formula~\cite{joyal1981theorie} and of the related proof due to E\u gecio\u glu and Remmel~\cite{egecioglu1986bijections}.  It can be seen as a specialization of Remmel and Williamson's proof of Theorem~2.4 in~\cite{remmel2002spanning}.  

In Section~\ref{sec:diff} the methodology from threshold graphs is applied to difference graphs, and the corresponding formula for the number of spanning trees in a difference graph is developed.  We conclude in Section~\ref{sec:disc} with discussion.

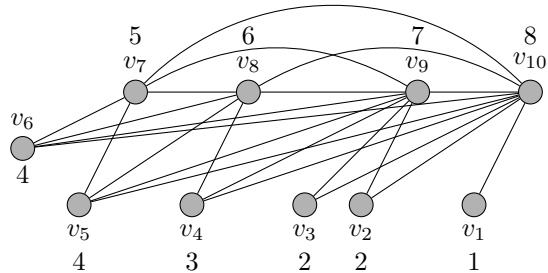
\begin{figure}[t]
\begin{center}
\begin{tikzpicture}[>=stealth,-,auto]
    \tikzstyle{every node}=[draw,circle,fill=black!30,minimum size=9pt,inner sep=0pt]
    \draw (0,0.5*\Gsc)      node (6) [label=above:$v_6$] {};
    \draw (0.5*\Gsc,0)      node (5) [label=below:$v_5$] {};
    \draw (1*\Gsc,1*\Gsc)   node (7) [label=above:$v_7$] {};
    \draw (1.5*\Gsc,0)      node (4) [label=below:$v_4$] {};
    \draw (2*\Gsc,1*\Gsc)   node (8) [label=above:$v_8$] {};
    \draw (2.5*\Gsc,0)      node (3) [label=below:$v_3$] {};
    \draw (3*\Gsc,0)        node (2) [label=below:$v_2$] {};
    \draw (3.5*\Gsc,1*\Gsc) node (9) [label=above:$v_9$] {};
    \draw (4*\Gsc,0)        node (1) [label=below:$v_1$] {};
    \draw (4.5*\Gsc,1*\Gsc) node (10) [label=above:$v_{10}$] {};

    \node at (0,0.4)               [draw=none,fill=none] {4};
    \node at (0.5*\Gsc,-0.5*\Gsc) [draw=none,fill=none] {4};
    \node at (1*\Gsc,1.5*\Gsc)   [draw=none,fill=none] {5};
    \node at (1.5*\Gsc,-0.5*\Gsc) [draw=none,fill=none] {3};
    \node at (2*\Gsc,1.5*\Gsc)   [draw=none,fill=none] {6};
    \node at (2.5*\Gsc,-0.5*\Gsc) [draw=none,fill=none] {2};
    \node at (3*\Gsc,-0.5*\Gsc)   [draw=none,fill=none] {2};
    \node at (3.5*\Gsc,1.5*\Gsc) [draw=none,fill=none] {7};
    \node at (4*\Gsc,-0.5*\Gsc)   [draw=none,fill=none] {1};
    \node at (4.5*\Gsc,1.5*\Gsc) [draw=none,fill=none] {8};

    \draw (7) -- (5);
    \draw (7) -- (6);
    \draw (8) -- (4);
    \draw (8) -- (7);
    \draw (8) -- (5);
    \draw (8) -- (6);
    \draw (9) -- (6);
    \draw (9) -- (5);
    \draw (9) to[bend right=30] (7);
    \draw (9) -- (4);
    \draw (9) -- (8);
    \draw (9) -- (3);
    \draw (9) -- (2);
    \draw (10) -- (6);
    \draw (10) -- (5);
    \draw (10) to[bend right=45] (7);
    \draw (10) -- (4);
    \draw (10) to[bend right=30] (8);
    \draw (10) -- (3);
    \draw (10) -- (2);
    \draw (10) -- (9);
    \draw (10) -- (1);
\end{tikzpicture}
\end{center}
\caption{A threshold graph with an integral weight assigned to each vertex.  Two vertices are adjacent if and only if the sum of their weights is at least 9.}\label{fig:thresholdgraph}
\end{figure}

\subsection{Theorem~\ref{thm:main} as a consequence of Merris' Theorem and the Matrix Tree Theorem \label{sec:cons}}

Suppose the graph~$G$ has vertices~$v_1,v_2,\ldots,v_n$. The \emph{Laplacian} of~$G$ is the matrix~${\cL \in \R^{n \times n}}$ such that every entry~$L_{ij}$ is $\de (v_i)$,~$-1$,~or~$0$ according as ${i=}j$,~${v_i \sim v_j}$,~or~neither. The Matrix Tree~Theorem can be stated in terms of submatrices of~$\cL$ or in terms of the eigenvalues of~$\cL$.  Since $\cL$ is symmetric, its eigenvalues can be ordered as ${\lambda_{1} \geq \lambda_{2} \geq \cdots \geq \lambda_{n}}$. It is easy to see (by the Gershgorin~Theorem and the fact that every row sum of~$\cL$ is zero) that~$\cL$ is positive semidefinite and singular, hence~$\lambda_{n}=0$.

\begin{theorem}[Kirchhoff~\cite{kirchhoff1847ueber}] For any graph~$G$ on $n$~vertices it holds that 
\[\tau (G)=\left|\det{\cL'}\right|=\frac{1}{n}\prod_{i=1}^{n-1}\lambda_{i},\]
where $\cL'$ is any $(n-1)\times(n-1)$ submatrix of $\cL$.
\end{theorem}
The earliest proofs of Cayley's~formula~\cite{sylvester1857on,borchardt1860uber} use the Matrix Tree~Theorem; indeed ${\cL_{K_n}=nI-J}$ where $I$ is the ${n \times n}$ identity matrix and $J$ is the ${n \times n}$ matrix of all ones.  The eigenvalues of~$\cL_{K_n}$ are easily seen to be ${n,n,\ldots,n,0}$; thus the Matrix Tree Theorem gives the number of spanning trees in $K_n$ as $\frac{1}{n}n^{n-1}=n^{n-2}$.

In general, the eigenvalues of a Laplacian matrix are not integers. However, for threshold graphs, not only are the eigenvalues integers but they are easily described with the graph's degree sequence.
\begin{theorem}[Merris~\cite{merris1994degree}] For any threshold graph~$G=(V,E)$ with $n$~vertices, it holds for each~$i=1,2,\ldots,n$~that 
\[\lambda_{i}=| \{ v\in V: \de (v) \geq i  \}|.\]
\end{theorem}

Theorem~\ref{thm:main} now follows immediately from the Matrix Tree~Theorem and Merris'~Theorem: 
\[\tau (G)= \frac{1}{n} \prod_{i=1}^{n-1}\lambda_{i}= \frac{1}{n}\prod_{i=1}^{n-1}| \{ v\in V: \de (v) \geq i  \}|.\] 

\subsection{Characterization and properties of threshold graphs \label{sec:proper}}

\begin{figure}
\begin{center}
\begin{tikzpicture}[>=stealth,-,auto]
    \tikzstyle{every node}=[draw,circle,fill=black!30,minimum size=9pt,inner sep=0pt]
    \draw (0,0.5*\Gsc)      node (6) [label=above:$v_6$] {};
    \draw (0.5*\Gsc,0)      node (5) [label=below:$v_5$] {};
    \draw (1*\Gsc,1*\Gsc)   node (7) [label=above:$v_7$] {};
    \draw (1.5*\Gsc,0)      node (4) [label=below:$v_4$] {};
    \draw (2*\Gsc,1*\Gsc)   node (8) [label=above:$v_8$] {};
    \draw (2.5*\Gsc,0)      node (3) [label=below:$v_3$] {};
    \draw (3*\Gsc,0)        node (2) [label=below:$v_2$] {};
    \draw (3.5*\Gsc,1*\Gsc) node (9) [label=above:$v_9$] {};
    \draw (4*\Gsc,0)        node (1) [label=below:$v_1$] {};
    \draw (4.5*\Gsc,1*\Gsc) node (10) [label=above:$v_{10}$] {};

    \node at (0,-0.5*\Gsc)        [draw=none,fill=none] {*};
    \node at (0.5*\Gsc,-0.5*\Gsc) [draw=none,fill=none] {0};
    \node at (1*\Gsc,-0.5*\Gsc)   [draw=none,fill=none] {1};
    \node at (1.5*\Gsc,-0.5*\Gsc) [draw=none,fill=none] {0};
    \node at (2*\Gsc,-0.5*\Gsc)   [draw=none,fill=none] {1};
    \node at (2.5*\Gsc,-0.5*\Gsc) [draw=none,fill=none] {0};
    \node at (3*\Gsc,-0.5*\Gsc)   [draw=none,fill=none] {0};
    \node at (3.5*\Gsc,-0.5*\Gsc) [draw=none,fill=none] {1};
    \node at (4*\Gsc,-0.5*\Gsc)   [draw=none,fill=none] {0};
    \node at (4.5*\Gsc,-0.5*\Gsc) [draw=none,fill=none] {1};
    \node at (2*\Gsc,-1*\Gsc)     [draw=none,fill=none] {};

    \draw (7) -- (5);
    \draw (7) -- (6);
    \draw (8) -- (4);
    \draw (8) -- (7);
    \draw (8) -- (5);
    \draw (8) -- (6);
    \draw (9) -- (6);
    \draw (9) -- (5);
    \draw (9) to[bend right=30] (7);
    \draw (9) -- (4);
    \draw (9) -- (8);
    \draw (9) -- (3);
    \draw (9) -- (2);
    \draw (10) -- (6);
    \draw (10) -- (5);
    \draw (10) to[bend right=45] (7);
    \draw (10) -- (4);
    \draw (10) to[bend right=30] (8);
    \draw (10) -- (3);
    \draw (10) -- (2);
    \draw (10) -- (9);
    \draw (10) -- (1);
\end{tikzpicture}
\end{center}
\caption{A threshold graph with the creation sequence $*010100101$.  The vertices appear from left to right according to their creation order with dominating vertices drawn raised in the figure and independent vertices drawn lowered in the figure.}\label{fig:creat}
\end{figure}
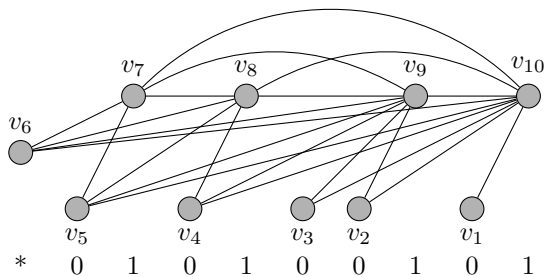

A \emph{creation sequence} is a string of characters $c_1,c_2,\ldots,c_n$ for some positive integer~$n$ such that $c_1$ is the character~``$*$" and, for each~$i>1$, $c_i$~is either the character~``$1$" or the character~``$0$". A creation sequence is viewed as a set of instructions for iteratively constructing a particular graph; specifically, the first character~``$*$" indicates that we begin with a single vertex, then iteratively, for each $i=2,3,\ldots,n$, a new vertex is added.  If~$c_i$ is~``$1$" then the $i^{th}$~vertex is added with edges to all of the $i-1$ currently-existing vertices; such a vertex will be called a \emph{dominating~vertex}.  If~$c_i$ is~``$0$" then the $i^{th}$~vertex is added without any edges; such a vertex will be called an \emph{independent~vertex}.  We adopt the naming convention that the first vertex (per~character~``$*$") is not ``dominating'' nor ``independent.'' If~$i<j$ then we say the vertex added per~$c_j$ is \emph{created~later} than the vertex added per~$c_i$. Figure~\ref{fig:creat} illustrates that the graph in Figure~\ref{fig:thresholdgraph} can be constructed according to the creation sequence~$*010100101$.

\begin{proposition}[Hagberg, Swart, and Schult \cite{hagberg2006designing}]
A graph~$G$ is a threshold graph if and only if there is a creation sequence that can construct it. \label{prop:creat}
\end{proposition}
Proposition~\ref{prop:creat} is easy to prove, here are the main ideas.  First, every threshold graph has a vertex adjacent to all the other vertices or has a vertex adjacent to none of the other vertices (according as the vertices of highest $\phi$-value are or are not adjacent to the vertices of lowest $\phi$-value).  Iterative removal of such vertices yields a creation sequence for the threshold graph in reverse order. Conversely, given a creation sequence, $\phi$-weights can be iteratively assigned to the vertices, as the creation sequence creates them, that are high (or low) enough to indicate all of the adjacencies (or non-adjacencies).

Per Proposition~\ref{prop:creat}, we view every threshold graph~$G=(V,E)$ as created by a creation sequence.  We let~$U$ denote the set of dominating vertices and $Z$ denote the set of independent vertices, so $V$ is partitioned into~$U$,~$Z$, and a singleton~$\{v_*\}$ containing the vertex created at the character~``$*$". 

For any graph~${G=(V,E)}$ and~$v \in V$ define~${N_G(v):=\{ u \in V: u \sim v \}}$ and ${N_G[v]:=N_G(v) \cup \{ v \}}$. 
\begin{proposition}[Chv\'atal and Hammer \cite{chvatal1975aggregation}]Suppose $G=(V,E)$ is a threshold graph on $n$~vertices. If the vertices~$V$ are labeled~$v_1,v_2,\ldots,v_n$ such that $\de (v_1) \leq \de (v_2) \leq \cdots \leq \de (v_n)$, then there is an integer~$m\geq0$ such that $Z=\{ v_1,v_2,\ldots,v_{m} \}$, $v_*=v_{m+1}$, and $U=\{v_{m+2}, \ldots, v_{n-1}, v_n \}$. Furthermore, 
\begin{align*}
N_G(v_i)&=\{ v \in V: \de (v)\geq n-i\}, &\text{for }i=1,2,\ldots,m\text{, and}\\
N_G[v_i]&=\{ v \in V: \de (v)\geq n-i+1 \}, &\text{for }i=m+2,m+3,\ldots,n. 
\end{align*}\label{prop:neigh}
\end{proposition}

Brief contemplation can informally illuminate the idea of the proof of Proposition~\ref{prop:neigh}.  The vertices of Figure~\ref{fig:creat} are photographically~reproduced in Figure~\ref{fig:neigh}, and a directed path is drawn through the vertices in the decreasing order of their degrees.  Corresponding to each vertex, a white~token is placed on the edge of this path directly above~or~below the vertex, and the token corresponding to the vertex~$v_*$ is placed on the vertex~itself.  The tokens encountered upon crossing an edge closely correspond to the difference in the adjacencies of each of its incident vertices.  In fact, for~each~${v\in U}$ observe that $N[v]$~corresponds exactly to the tokens encountered by starting at~$v$ and following the directed path, and for~each~${v\in Z}$ the tokens encountered along the path from~$v$ correspond~to~$N(v)$.

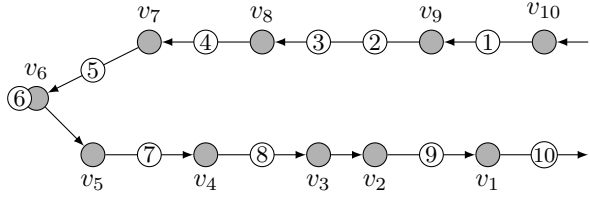
\begin{figure}
\begin{center}
\begin{tikzpicture}[-,auto]
    \tikzstyle{every node}=[draw,circle,fill=black!30,minimum size=9pt,inner sep=0pt]
    \draw (0,0.5*\Gsc)      node (6)  [label=above:$v_6$]{};
    \draw (0.5*\Gsc,0)      node (5)  [label=below:$v_5$]{};
    \draw (1*\Gsc,1*\Gsc)   node (7)  [label=above:$v_7$]{};
    \draw (1.5*\Gsc,0)      node (4)  [label=below:$v_4$]{};
    \draw (2*\Gsc,1*\Gsc)   node (8)  [label=above:$v_8$]{};
    \draw (2.5*\Gsc,0)      node (3)  [label=below:$v_3$]{};
    \draw (3*\Gsc,0)        node (2)  [label=below:$v_2$]{};
    \draw (3.5*\Gsc,1*\Gsc) node (9)  [label=above:$v_9$]{};
    \draw (4*\Gsc,0)        node (1)  [label=below:$v_1$]{};
    \draw (4.5*\Gsc,1*\Gsc) node (10) [label=above:$v_{10}$]{};
    \draw (5*\Gsc,0)        node (0) [draw=none,fill=none]{};
    \draw (5*\Gsc,1*\Gsc)        node (11) [draw=none,fill=none]{};

    \draw [>=latex,->](11) -- (10);
    \draw [>=latex,->](10) -- (9);
    \draw [>=latex,->](9) -- (8);
    \draw [>=latex,->](8) -- (7);
    \draw [>=latex,->](7) -- (6);
    \draw [>=latex,->](6) -- (5);
    \draw [>=latex,->](5) -- (4);
    \draw [>=latex,->](4) -- (3);
    \draw [>=latex,->](3) -- (2);
    \draw [>=latex,->](2) -- (1);
    \draw [>=latex,->](1) -- (0);

    \node at (-0.15*\Gsc,0.5*\Gsc) [fill=white] {\small 6};
    \node at (0.5*\Gsc,0.75*\Gsc) [fill=white] {\small 5};
    \node at (1*\Gsc,0*\Gsc)   [fill=white] {\small 7};
    \node at (1.5*\Gsc,1*\Gsc) [fill=white] {\small 4};
    \node at (2*\Gsc,0*\Gsc)   [fill=white] {\small 8};
    \node at (2.5*\Gsc,1*\Gsc) [fill=white] {\small 3};
    \node at (3*\Gsc,1*\Gsc)   [fill=white] {\small 2};
    \node at (3.5*\Gsc,0*\Gsc) [fill=white] {\small 9};
    \node at (4*\Gsc,1*\Gsc)   [fill=white] {\small 1};
    \node at (4.5*\Gsc,0)      [fill=white] {\footnotesize 10};
\end{tikzpicture}
\end{center}
\caption{An illustration of Proposition~\ref{prop:neigh} with the threshold graph in Figure~\ref{fig:creat}.  Begin at any vertex~$v$ and follow the arrows anti-clockwise.  The white tokens encountered name either the vertices of $N[v]$ or those of $N(v)$ according as $v\in U$ or $v\in Z$.}\label{fig:neigh}
\end{figure}

A consequence of Proposition~\ref{prop:neigh} is that the following theorem is equivalent to Theorem~\ref{thm:main}.  In fact, Hammer and Kelmans~\cite{hammer1996laplacian} state the theorem in this form.

\begin{theorem} \label{thm:main2} For any threshold graph~$G=(V,E)$ with $n$~vertices it holds that $$\tau(G)=\frac{1}{n}\left ( \prod_{v \in U}(\de (v)+1)  \right )\left (  \prod_{v \in Z} \de (v) \right ).$$
\end{theorem}

To see that Theorem~\ref{thm:main2} is equivalent to Theorem~\ref{thm:main}, first  let the vertices be labeled~${v_1,v_2,\ldots,v_n}$ as in Proposition~\ref{prop:neigh}. By Proposition~\ref{prop:neigh}, the number of vertices of degree at least~$1,2,\ldots,|U|$ are, respectively, $|N[v_n]|$,~$|N[v_{n-1}]|$,~\ldots,~$|N[v_{n-|U|+1}]|$ and the number of vertices of degree at least $|U|+1,|U|+2,\ldots,n-2,n-1$ are, respectively, $|N(v_{|Z|})|$,~\ldots,~$|N(v_2)|$,~$|N(v_1)|$.  Multiplying these yields 
\[\frac{1}{n}\prod_{i=1}^{n-1} | \{ v\in V: \de (v) \geq i  \} |=
\frac{1}{n}\left ( \prod_{v \in U}|N_G[v]|  \right )
\left (  \prod_{v \in Z} |N_G(v)| \right );\] 
hence Theorem~\ref{thm:main2} is equivalent to Theorem~\ref{thm:main}.

\section{A combinatorial proof of Theorem~\ref{thm:main} \label{sec:comb}}

In this section we provide a combinatorial proof for Theorem~\ref{thm:main2}, hence a combinatorial proof of Theorem~\ref{thm:main}. It is an extension of Joyal's proof of Cayley's formula~\cite{joyal1981theorie} and the related proof by E\u gecio\u glu and Remmel~\cite{egecioglu1986bijections}.  The proof can be viewed as a special case of a proof due to Remmel and Williamson~\cite{remmel2002spanning}.

If $|V|=1$ the result is immediate.  Fix a connected threshold graph~${G=(V,E)}$ with at least two vertices.  Partition~$V$ into~$U$,~$Z$, and~$\{v_*\}$ according to the creation sequence of~$G$. Define~$\cF$ to be the set of all functions~${f:V \rightarrow V}$ such that ${f(v)\in N(v)}$~for~all~${v \in Z\cup\{v_*\}}$, and ${f(v) \in N[v]}$~for~all~${v \in U}$.  Of course~${N(v_*)=U}$; hence, we have 
\[|\cF|= |U| \left ( \prod_{v \in U}(\de (v)+1)  \right )\left (  \prod_{v \in Z} \de (v) \right ).\] 

Define~$\cT$ to be the set of spanning trees of~$G$ with one vertex in~$U$ ``marked~black'' and one vertex in $V$ ``marked~white'' (possibly the same vertex that is marked black), so that~$|\cT|=|U| \cdot |V| \cdot \tau (G)$.  In what follows we define a one-to-one correspondence~$\Psi:\cF \rightarrow\cT$; when this is accomplished we are done, since 
\[\tau(G)= \frac{1}{| U | \cdot |V| } |\cT|=\frac{1}{| U | \cdot |V| } | \cF | = \frac{1}{|V|} \prod_{v \in U}(\de (v)+1) \prod_{v \in Z} \de (v)\]
yields Theorem~\ref{thm:main2}.

Let $f$ be any function in $\cF$; here we describe how to find the marked tree $\Psi(f)$. The function~$f$ is uniquely identified with a directed graph~$D$ whose vertices are~$V$, and each ordered~pair~$(u,v) \in V\times V$ is a directed edge of $D$ if and only if~$f(u)=v$.  Notice that each edge of $D$ is an (undirected) edge of~$G$ or a loop.

For a moment, treat~$D$ as an undirected graph (it may have loops and multiple edges) and let~$\{ V_j \}_{j=1}^k$ be the partition of~$V$ according to connected components of the undirected~$D$.  

Now back to the directed graph~$D$, the out-degree of every vertex is exactly~one; hence, each induced subgraph~$D(V_j)$ contains exactly one directed cycle, call it~$C_j$.  For example, taking the threshold graph $G$ in Figures~\ref{fig:thresholdgraph} and~\ref{fig:creat} and the function 
\begin{equation}\label{eq:examplef}
f=\left(\begin{array}{cccccccccc}
v_1&v_2&v_3&v_4&v_5&v_6&v_7&v_8&v_9&v_{10}\\
v_{10}&v_9&v_9&v_8&v_7&v_7&v_8&v_{5}&v_{10}&v_{10}\end{array}\right)
\end{equation}
leads to the directed graph in Figure~\ref{fig:treeconstruction:a}.  That graph has two~cycles.

\begin{figure}[h!]
\begin{center}
\subfigure[]{\label{fig:treeconstruction:a}
\begin{tikzpicture}[>=stealth,-,auto]
    \tikzstyle{every node}=[draw,circle,fill=black!20,minimum size=9pt,inner sep=0pt]

    \draw (0,0.5*\Gsc)      node (6) [label=left:$v_6$] {};
    \draw (0.5*\Gsc,0)      node (5) [label=below:$v_5$] {};
    \draw (1*\Gsc,1*\Gsc)   node (7) [label=above:$v_7$] {};
    \draw (1.5*\Gsc,0)      node (4) [label=below:$v_4$] {};
    \draw (2*\Gsc,1*\Gsc)   node (8) [label=above:$v_8$] {};
    \draw (2.5*\Gsc,0)      node (3) [label=below:$v_3$] {};
    \draw (3*\Gsc,0)        node (2) [label=below:$v_2$] {};
    \draw (3.5*\Gsc,1*\Gsc) node (9) [label=above:$v_9$] {};
    \draw (4*\Gsc,0)        node (1) [label=below:$v_1$] {};
    \draw (4.5*\Gsc,1*\Gsc) node (10) [label=above:$v_{10}$] {};

    \draw[arrows={-triangle 45}] (7) -- (8);
    \draw[arrows={-triangle 45}] (6) -- (7);
    \draw[arrows={-triangle 45}] (8) -- (5);
    \draw[arrows={-triangle 45}] (5) -- (7);
    \draw[arrows={-triangle 45}] (4) -- (8);
    \draw[arrows={-triangle 45}] (9) -- (10);
    \draw[arrows={-triangle 45}] (1) -- (10);
    \draw[arrows={-triangle 45}] (10) to[loop,in=-30,out=30,looseness=10] (10);
    \draw[arrows={-triangle 45}] (3) -- (9);
    \draw[arrows={-triangle 45}] (2) -- (9);
\end{tikzpicture}}
\subfigure[]{\label{fig:treeconstruction:b}
\begin{tikzpicture}[>=stealth,-,auto]

    \draw (4.5*\Gsc,1*\Gsc) node (10) [draw,label=above:$v_{10}$,minimum size=8pt,inner sep=0pt,shape=circle,fill=white] {};

    \tikzstyle{every node}=[draw,circle,fill=black!20,minimum size=9pt,inner sep=0pt]

    \draw (0,0.5*\Gsc)      node (6) [label=left:$v_6$] {};
    \draw (0.5*\Gsc,0)      node (5) [label=below:$v_5$] {};
    \draw (1*\Gsc,1*\Gsc)   node (7) [label=above:$v_7$] {};
    \draw (1.5*\Gsc,0)      node (4) [label=below:$v_4$] {};
    \draw (2*\Gsc,1*\Gsc)   node (8) [fill=black,label=above:$v_8$] {};
    \draw (2.5*\Gsc,0)      node (3) [label=below:$v_3$] {};
    \draw (3*\Gsc,0)        node (2) [label=below:$v_2$] {};
    \draw (3.5*\Gsc,1*\Gsc) node (9) [label=above:$v_9$] {};
    \draw (4*\Gsc,0)        node (1) [label=below:$v_1$] {};

    \draw[arrows={-triangle 45}] (7) -- (8);
    \draw[arrows={-triangle 45}] (8) -- (5);
    \draw[arrows={-triangle 45}] (5) -- (7);
    \draw[arrows={-triangle 45}] (6) -- (7);
    \draw[arrows={-triangle 45}] (4) -- (8);
    \draw[arrows={-triangle 45}] (9) -- (10);
    \draw[arrows={-triangle 45}] (1) -- (10);
    \draw[arrows={-triangle 45}] (10) to[loop,in=-30,out=30,looseness=10] (10);
    \draw[arrows={-triangle 45}] (3) -- (9);
    \draw[arrows={-triangle 45}] (2) -- (9);
\end{tikzpicture}}
\subfigure[]{\label{fig:treeconstruction:c}
\begin{tikzpicture}[>=stealth,-,auto]

    \draw (4.5*\Gsc,1*\Gsc) node (10) [draw,label=above:$v_{10}$,minimum size=8pt,inner sep=0pt,shape=circle,fill=white] {};
    \tikzstyle{every node}=[draw,circle,fill=black!20,minimum size=9pt,inner sep=0pt]

    \draw (0,0.5*\Gsc)      node (6) [label=left:$v_6$] {};
    \draw (0.5*\Gsc,0)      node (5) [label=below:$v_5$] {};
    \draw (1*\Gsc,1*\Gsc)   node (7) [label=above:$v_7$] {};
    \draw (1.5*\Gsc,0)      node (4) [label=below:$v_4$] {};
    \draw (2*\Gsc,1*\Gsc)   node (8) [fill=black,label=above:$v_8$] {};
    \draw (2.5*\Gsc,0)      node (3) [label=below:$v_3$] {};
    \draw (3*\Gsc,0)        node (2) [label=below:$v_2$] {};
    \draw (3.5*\Gsc,1*\Gsc) node (9) [label=above:$v_9$] {};
    \draw (4*\Gsc,0)        node (1) [label=below:$v_1$] {};

    \draw[arrows={-triangle 45}] (7) to[bend left=30] (10);
    \draw[arrows={-triangle 45}] (8) -- (5);
    \draw[arrows={-triangle 45}] (5) -- (7);
    \draw[arrows={-triangle 45}] (6) -- (7);
    \draw[arrows={-triangle 45}] (4) -- (8);
    \draw[arrows={-triangle 45}] (9) -- (10);
    \draw[arrows={-triangle 45}] (1) -- (10);
    \draw[arrows={-triangle 45}] (3) -- (9);
    \draw[arrows={-triangle 45}] (2) -- (9);
\end{tikzpicture}}
\subfigure[]{\label{fig:treeconstruction:d}
\begin{tikzpicture}[>=stealth,-,auto]
    \tikzstyle{every node}=[draw,circle,fill=black!20,minimum size=9pt,inner sep=0pt]

    \draw (0,0.5*\Gsc)      node (6) [label=left:$v_6$] {};
    \draw (0.5*\Gsc,0)      node (5) [label=below:$v_5$] {};
    \draw (1*\Gsc,1*\Gsc)   node (7) [label=above:$v_7$] {};
    \draw (1.5*\Gsc,0)      node (4) [label=below:$v_4$] {};
    \draw (2*\Gsc,1*\Gsc)   node (8) [fill=black,label=above:$v_8$] {};
    \draw (2.5*\Gsc,0)      node (3) [label=below:$v_3$] {};
    \draw (3*\Gsc,0)        node (2) [label=below:$v_2$] {};
    \draw (3.5*\Gsc,1*\Gsc) node (9) [label=above:$v_9$] {};
    \draw (4*\Gsc,0)        node (1) [label=below:$v_1$] {};
    \draw (4.5*\Gsc,1*\Gsc) node (10) [fill=white,label=above:$v_{10}$] {};

    \draw (7) to[bend left=30] (10);
    \draw (8) -- (5);
    \draw (5) -- (7);
    \draw (6) -- (7);
    \draw (4) -- (8);
    \draw (9) -- (10);
    \draw (1) -- (10);
    \draw (3) -- (9);
    \draw (2) -- (9);
\end{tikzpicture}}
\end{center}
\caption{An illustration of the steps in the combinatorial proof of Theorem~\ref{thm:main}.  \subref{fig:treeconstruction:a} shows the directed graph for the function~\eqref{eq:examplef}; it corresponds to the marked spanning tree drawn in \subref{fig:treeconstruction:d} of the graph in Figures~\ref{fig:thresholdgraph} and~\ref{fig:creat}.}
\end{figure}
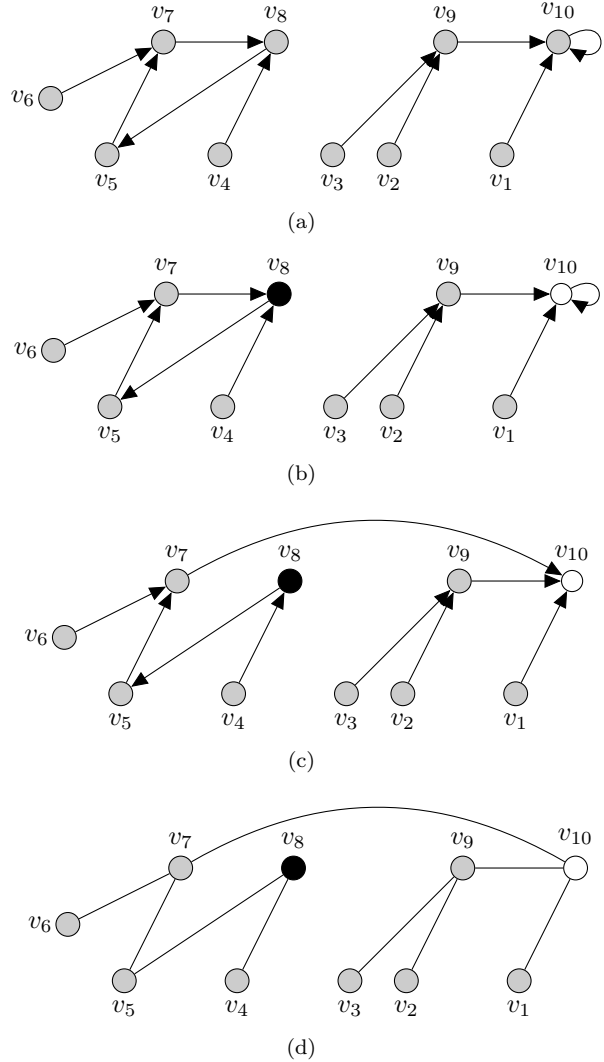

For each~$j$, let~$\ell_j$ denote the vertex created last among the vertices of the cycle~$C_j$ and let $b_j$ denote the vertex on the same cycle satisfying $f(b_j)=\ell_j$; hence, $(b_j,\ell_j)$ is an edge of $D$ that lies on the cycle $C_j$.  Notice that each vertex~$\ell_j$ is a member of $U$ either because it is adjacent to and created later than~$b_j$ or because~$f(\ell_j)=\ell_j$.

Without loss of generality, assume that~$\ell_1,\ell_2,\ldots,\ell_k$ are in creation order; i.e.~$\ell_2$ was created after~$\ell_1$, $\ell_3$ was created after~$\ell_2$, etc. Color~$\ell_1$~black and $b_k$~white.  In our example, $(b_1,\ell_1)$~is~$(v_7,v_8)$ and $(b_2,\ell_2)$~is~$(v_{10},v_{10})$; the vertices are colored for Figure~\ref{fig:treeconstruction:b}.  

The next step is to connect $\ell_1$~and~$b_k$ with a path by adding the $k-1$~edges $\{(b_j,\ell_{j+1})\}_{j=1}^{k-1}$ to~$D$ and deleting the $k$~edges $\{(b_j,\ell_j)\}_{j=1}^k$, as in Figure~\ref{fig:treeconstruction:c}.  The vertices $\ell_1,\ell_2,\ldots,\ell_k$ are members of $U$ ordered according to the creation sequence; thus, $\ell_{j+1}$ is created later than $b_j$, which implies that $b_j\ell_{j+1}$ is an edge of $G$ for~$j=1,\ldots,k-1$.  Notice that, after the additions and deletions, $D$ is now a tree with every edge directed toward $b_k$, the white vertex.  The final step is to remove the directions from the edges, as in Figure~\ref{fig:treeconstruction:d}.

Clearly, the graph~$\Psi (f)$ is a member of~$\cT$.  It is a connected subgraph of~$G$ with~$|V|-1$ edges, and the vertices $\ell_1\in U$ and $b_k$ are marked black and white, respectively. In our example, $\Psi (f)$ is the spanning tree in Figure~\ref{fig:treeconstruction:d}, with $v_8$~marked~black and $v_{10}$~marked~white.

Now, we define a function~$\Psi^-: \cT \rightarrow \cF$; it will turn out to be the inverse of~$\Psi$.  Suppose we are given any marked spanning tree~$T$ of~$G$.  First, direct every edge of~$T$ toward the white~vertex and call this directed graph~$D$. There is a unique path~$\nu_1,\nu_2,\ldots,\nu_m$ in~$D$ with $\nu_1$~marked black and $\nu_m$~marked white.  

Let~$\ell_1,\ell_2,\ldots,\ell_k$ be the longest subsequence of~$\nu_1,\nu_2,\ldots,\nu_m$ such that every vertex $\ell_j$ is created later than all vertices that precede it along the path; i.e. if $\ell_j=\nu_i$ and $i'<i$ then $\ell_j$ is created later than $\nu_{i'}$.  Of course, $\ell_1=\nu_1$ so the subsequence is not empty, and it is not hard to show that the subsequence is unique.  By construction, the vertices~$\ell_1,\ell_2,\ldots,\ell_k$ are in creation order.  Also, $\ell_k$ is created later than every other vertex on the path and, for $j=1,2,\ldots,k-1$, the vertex $\ell_j$ is created later than every other vertex preceding $\ell_{j+1}$ along the path.  Additionally, the vertex $\ell_1$ is a member of $U$ because it is marked black and the vertices $\ell_2,\ldots,\ell_k$ are members of $U$ because each is created later than its neighbor that precedes it on the path.

For example, if~$T$ is the marked tree in Figure~\ref{fig:treeconstruction:d}, then directing the edges gives the graph in Figure~\ref{fig:treeconstruction:c}.  The path~$\nu_1,\nu_2,\ldots,\nu_m$ is~$v_8,v_5,v_7,v_{10}$, and the identified vertices are~$\ell_1=v_8$ and $\ell_2=v_{10}$.

The next step is to remove the $k-1$ edges along the path that are directed into $\ell_2,\ell_3,\ldots,\ell_{k}$.  This breaks $D$ into a directed forest where each of the $k$ directed trees contains a segment of the path discussed in the last paragraphs; each of the vertices $\ell_1,\ell_2,\ldots,\ell_k$ is an endpoint of one path segment.  The final step is to add~$k$ directed edges, one edge between the endpoints of each path segment, to form~$k$ directed cycles.  Recall that, for~each~$j$, the vertex~$\ell_j$ is created later than the other vertices along its segment and $\ell_j$ is in $U$; thus each added edge is also an (undirected) edge of~$G$ or a loop on a vertex in $U$.

This manner of manipulating the edges in~$D$ enforces that, after the additions, every vertex in~$D$ has out-degree one.  Thus, $D$ represents a function~$f:V\to V$.  Since~$T$ is a spanning tree and by our careful choice of the sequence~$\{\ell_j\}$, we have ensured that every directed edge of $D$ is also an undirected edge of~$G$ or is a loop on a vertex in~$U$.  Therefore, $f\in\cF$.

By the definitions of $\Psi$ and $\Psi^-$, and with the observations
we made along the way, it is clear that $\Psi^- ( \Psi (f))=f$, for all $f \in \cF$,  and $\Psi( \Psi^- (T))=T$, for all $T \in \cT$; therefore $\Psi:\cF \rightarrow \cT$ is bijective (with inverse $\Psi^-$), as desired.

\section{Applying the methodology to difference graphs \label{sec:diff}}

In this section we apply the threshold graph methodology developed above to investigate a related class of graphs.  The result is a spanning tree counting formula for difference graphs that is analogous to the threshold graph formula.

A bipartite graph $H=(X,Y,E)$ is a \emph{difference graph} if there exists a function $\phi:X \cup Y  \rightarrow \R$ and a threshold $\alpha \in \R$, such that for all $x \in X,y \in Y$ it holds that $x \sim y$ if and only if $\phi (x) + \phi (y) \geq \alpha$.

Difference graphs are close cousins of threshold graphs, and many threshold graph theorems have adaptations to difference graphs; see for example  \cite{mahadev1995threshold,ross2011properties}.  The following spanning tree formula for difference graphs is obviously analogous to Theorem~\ref{thm:main}.  
\begin{theorem}[Ehrenborg and van~Willigenburg~\cite{ehrenborg2004enumerative}] \label{thm:main3}
For any difference graph $H=(X,Y,E)$, if $X$ and $Y$ are nonempty then
$$\tau(H)=\frac{1}{|X||Y|}\prod_{i=1}^{|X|}
| \{ y\in Y: \de (y) \geq i  \} | \cdot \prod_{i=1}^{|Y|}
| \{ x\in X: \de (x) \geq i  \} |.$$
\end{theorem}

Theorem~\ref{thm:main3} has been proved using algebraic methods by Ehrenborg and van~Willigenburg~\cite{ehrenborg2004enumerative} and combinatorially with rook placements by Burns~\cite{burns2003bijective}.  In what follows, we prove Theorem~\ref{thm:main3} using the combinatorial machinery from the proof of Theorem~\ref{thm:main}.  The algebraic proof of Theorem~\ref{thm:main3} does not come as easily as for Theorem~\ref{thm:main}.  In fact, the eigenvalues of a difference graph are not generally integers, as is the case for threshold graphs.  It is our understanding that there is currently no analogue to Merris'~Theorem that yields a linear-algebraic proof of Theorem~\ref{thm:main3}.  

A \emph{bipartite creation sequence} is a string of characters~$c_1, c_2,\dots, c_n$ for some positive integer~$n$, in which all characters are either~``$0$"~or~``$1$". A bipartite creation sequence is viewed as a set of instructions for iteratively constructing a particular bipartite graph $(X,Y,E)$; begin with an empty graph and, for each $i=1,2,\ldots,n$, add a new vertex. If $c_i$ is ``$1$'' then the $i^{th}$~new vertex is added to the partite set $X$ with an edge to every vertex currently in $Y$.  Otherwise,  if $c_i$ is ``$0$'' then the $i^{th}$~new vertex is added to $Y$ with no edges.  As with threshold graphs, if ${i<j}$ then we say the vertex associated with $c_j$ is \emph{created later} than the vertex associated with $c_i$.  

Figure~\ref{fig:diffgraph} illustrates the bipartite graph created by the bipartite creation sequence $0010100101$. The vertices associated with character ``$1$", the partite set $X$, are drawn raised in the figure, vertices associated with character ``$0$", the partite set $Y$, are drawn lowered in the figure, and for all $i<j$ the vertices associated with $c_i$ and $c_j$ are adjacent if and only if $c_i$ is ``$0$'' and $c_j$ is ``$1$''. 

\begin{figure}
\begin{center}
\begin{tikzpicture}[>=stealth,-,auto]
    \tikzstyle{every node}=[draw,circle,fill=black!30,minimum size=9pt,inner sep=0pt]
    \node at (0,-0.5*\Gsc)        [draw=none,fill=none] {0};
    \node at (0.5*\Gsc,-0.5*\Gsc) [draw=none,fill=none] {0};
    \node at (1*\Gsc,-0.5*\Gsc)   [draw=none,fill=none] {1};
    \node at (1.5*\Gsc,-0.5*\Gsc) [draw=none,fill=none] {0};
    \node at (2*\Gsc,-0.5*\Gsc)   [draw=none,fill=none] {1};
    \node at (2.5*\Gsc,-0.5*\Gsc) [draw=none,fill=none] {0};
    \node at (3*\Gsc,-0.5*\Gsc)   [draw=none,fill=none] {0};
    \node at (3.5*\Gsc,-0.5*\Gsc) [draw=none,fill=none] {1};
    \node at (4*\Gsc,-0.5*\Gsc)   [draw=none,fill=none] {0};
    \node at (4.5*\Gsc,-0.5*\Gsc) [draw=none,fill=none] {1};

    \draw (0,0)             node (1) [label=below:$y_6$] {};
    \draw (0.5*\Gsc,0)      node (2) [label=below:$y_5$] {};
    \draw (1*\Gsc,1*\Gsc)   node (3) [label=above:$x_1$] {};
    \draw (1.5*\Gsc,0)      node (4) [label=below:$y_4$] {};
    \draw (2*\Gsc,1*\Gsc)   node (5) [label=above:$x_2$] {};
    \draw (2.5*\Gsc,0)      node (6) [label=below:$y_3$] {};
    \draw (3*\Gsc,0)        node (7) [label=below:$y_2$] {};
    \draw (3.5*\Gsc,1*\Gsc) node (8) [label=above:$x_3$] {};
    \draw (4*\Gsc,0)        node (9) [label=below:$y_1$] {};
    \draw (4.5*\Gsc,1*\Gsc) node (10) [label=above:$x_4$] {};

    \draw (3) -- (2);
    \draw (3) -- (1);
    \draw (5) -- (4);
    \draw (5) -- (2);
    \draw (5) -- (1);
    \draw (8) -- (1);
    \draw (8) -- (2);
    \draw (8) -- (4);
    \draw (8) -- (6);
    \draw (8) -- (7);
    \draw (10) -- (1);
    \draw (10) -- (2);
    \draw (10) -- (4);
    \draw (10) -- (6);
    \draw (10) -- (7);
    \draw (10) -- (9);
\end{tikzpicture}
\end{center}
\caption{A difference graph associated with the bipartite creation sequence $0010100101$. The vertices are arranged left-to-right according to the order of the creation sequence with vertices in $X$ drawn raised in the figure and vertices in $Y$ drawn lowered in the figure.
 \label{fig:diffgraph}}
\end{figure}
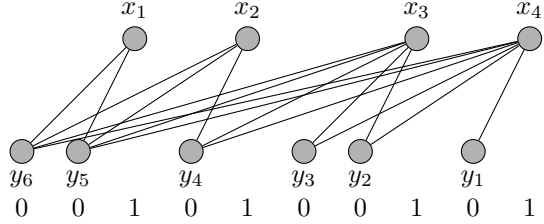

\begin{proposition}[Ross \cite{ross2011properties}] A graph $H$ is a difference graph if and only if there is a bipartite creation sequence that can construct it. \label{prop:condiff}
\end{proposition}

The proof of Proposition~\ref{prop:condiff} is much like the proof of Proposition~\ref{prop:creat}.  Suppose that~${H=(X,Y,E)}$ is a difference graph with weights~$\phi$ and threshold~$\alpha$ then, according as ${\max_{v\in X}\phi(v) + \min_{v\in Y}\phi(v)}$ is at least $\alpha$ or less than $\alpha$, either $X$ has a vertex adjacent to every vertex in $Y$ or $Y$ has an isolated vertex.  As before, iteratively removing those vertices gives a creation sequence for $H$ in reverse order.  Conversely, given a bipartite creation sequence the function $\phi$ can be defined iteratively, similarly to the earlier proposition. Per Proposition~\ref{prop:condiff}, we view every difference graph $H$ as created by a bipartite creation sequence.

\begin{proposition}[Mahadev and Peled \cite{mahadev1995threshold}, Theorem 2.4.4] \label{prop:bnd} Suppose $H=(X,Y,E)$ is a difference graph. If the vertices in $X$ are labeled $x_1,x_2,\ldots,x_{|X|}$ such that $\deg (x_1) \leq \deg (x_2) \leq \cdots \leq \deg (x_{|X|})$ and the vertices in $Y$ labeled  $y_1,y_2,\ldots,y_{|Y|}$ such that $\deg (y_1) \leq \deg (y_2) \leq \cdots \leq \deg (y_{|Y|})$ then, 
\begin{align*}
N(x_i)&=\{ y \in Y: \deg(y)\geq |X|-i+1\}, &\text{for }i=1,\ldots,|X|\text{, and}\\
N(y_i)&=\{ x \in X: \deg(x)\geq |Y|-i+1 \}, &\text{for }i=1,\ldots,|Y|. 
\end{align*}
\end{proposition}

This proposition can be proved directly.  It also can be seen from Proposition~\ref{prop:neigh} upon recognizing that if we form a new graph from $H$ by adding every edge with both endpoints in $X$ this new graph is a threshold graph with the same creation sequence, except for the initial character.

Proposition~\ref{prop:bnd} suggests restating Theorem~\ref{thm:main3} as follows.
\begin{theorem} \label{thm:main4} For any difference graph $H=(X,Y,E)$, if $X$ and $Y$ are nonempty then
$$  \tau (H)  =  \frac{1}{|X| \cdot |Y|}\prod_{v \in X \cup Y} \deg (v).$$
\end{theorem}

We now give a combinatorial proof for Theorem~\ref{thm:main4}, which provides a combinatorial proof for Theorem~\ref{thm:main3}.

\begin{proof}[Proof of Theorem~\ref{thm:main4}]
Let $H=(X,Y,E)$ be a difference graph.  $H$ is disconnected if and only if it has an isolated vertex (since $H$ is connected if and only if its bipartite creation sequence begins with ``0'' and ends with ``1''), so the formula holds when ${\tau(H)=0}$. 

Now, suppose $H$ is connected and take its creation sequence.  Let $G$ be the threshold graph with the same creation sequence as $H$, except that the first character, which is ``0'', is replaced by ``$*$''. 

Let $\cF$ be the set of functions~${f:X\cup Y\to X\cup Y}$ such that ${f(x)\in N_G[x]}$ for~${x\in X}$ and~${f(y)\in N_G(y)}$ for~${y\in Y}$, and let ${\H\subseteq\cF}$ be the set of functions~${f:X\cup Y\to X\cup Y}$ such that ${f(v)\in N_H(v)}$ for all~${v\in X\cup Y}$.  Clearly, $|\H|=\prod_{v\in X\cup Y}\de(v)$.

We show that $\Psi_G(\H)\subseteq \cT_H$ and~$\Psi^{-}_G(\cT_H)\subseteq\H$, where $\cT_H$ is the set of spanning trees of~$H$ with one vertex in~$X$ marked black and one vertex in~$Y$ marked white.  This proves that $|\H|=|\cT_H|$ and establishes the theorem, because $|\cT_H|=|X||Y|\tau(H)$ and $\Psi$ is a bijection with inverse~$\Psi^-$ .

Clearly, $\Psi_G(\H)$ is a set of spanning trees of~$G$.  It remains to show that each is a subgraph of~$H$.  Suppose $f\in\H$ and~$D$ is the corresponding directed graph.  Each edge in~$D$ is an undirected edge in~$H$ by construction.  $H$ is bipartite, so every cycle in~$D$ alternates in $X$ and $Y$.  Since the latest created vertex in each cycle is a vertex in~$X$, it holds that each edge added to connect the cycles of~$D$ connects a vertex in~$Y$ with a vertex in~$X$ and is an (undirected) edge of~$H$.  Thus,~$\Psi_G(f)\in\cT_H$.

Now suppose $T\in\cT_H$ is a marked spanning tree of~$H$.  By applying $\Psi_{G}^-$ and the logic of the last paragraph, we easily find that the vertices along the path from the black vertex to the white vertex alternate partite sets.  The edges removed from this path are each directed toward a vertex in~$X$, and since the first vertex is in~$X$ (because it is marked black) the endpoints of each segment of the broken path lie in different partite sets and are adjacent in $H$.  Thus,~$\Psi^{-}_G(T)\in\H$. \end{proof}

\section{Discussion \label{sec:disc}}

The many proofs of Cayley's formula are remarkable for their diversity.  The first proofs seem to have been given by Sylvester in 1857~\cite{sylvester1857on} and Borchardt in 1860~\cite{borchardt1860uber}.  Both proofs use Kirchhoff's Matrix Tree Theorem~\cite{kirchhoff1847ueber}, which was published in 1847.

The century and a half following Kirchhoff's paper brought more interest in Cayley's formula.  Moon~\cite{moon1967various} chronicles nine additional proofs published before~1967.  Five of the proofs in Moon's survey approach the problem with generating functions: P\'olya~\cite{pólya1937kombinatorische}, Dziobek~\cite{dziobek1947formel}, Katz~\cite{katz1955probability}, G\"obel~\cite{gobel1963mutual}, and Moon, himself~\cite{moon1963another}.  Papers published by Clarke~\cite{clarke1958cayley} and R\'enyi~\cite{renyi1959some} prove the theorem by deriving recurrences from tree properties.  Shor's 1995 proof~\cite{shor1995new} also uses a recurrence.

As for direct combinatorial proofs, there are two bijective proofs among the early papers:  Cayley's polynomials~\cite{cayley1889theorem} and Pr\"ufer's sequences~\cite{prufer1918neuer}.  Actually, Cayley's proof is only written for the case~$n=6$, and his paper states that the proof is ``applicable for any value whatever of~$n$''.   R\'enyi~\cite{renyi1970enumeration} published a proof for general~$n$, using Cayley's argument, in~1970.  Two more proofs rely on bijections between functions and trees; one is due to Joyal~\cite{joyal1981theorie}\footnote{Joyal's proof is given in English by Aigner and Ziegler in \emph{Proofs from The Book}~\cite[p.~202]{aigner2010proofs}.} and the other is due to E\u gecio\u glu and Remmel~\cite{egecioglu1986bijections}.  Pitman has a recent proof~\cite{pitman1999coalescent} by double counting.

The proofs of Theorems~\ref{thm:main}~and~\ref{thm:main3} can be derived from a proof by Remmel and Williamson \cite{remmel2002spanning}.  Indeed, one may view the proof of their Theorem 2.4 as a generalisation of Joyal's proof of Cayley's formula.  We believe that it is an interesting and worthwhile pursuit to generalize all of the proofs of Cayley's formula to larger classes of graphs, including threshold graphs.

\section{Acknowledgements}

The authors are grateful to Caroline Klivans, Victor Reiner, and Michael Slone for supplying additional references for this paper.

Stephen Chestnut was supported by U.S. Department of Education GAANN grant P200A090128.

\bibliographystyle{abbrv}
\bibliography{SpanningTreesOfThresholdGraphs}
\end{document}